\documentclass[10pt,journal,twocolumn,oneside]{IEEEtran}
\usepackage{ifpdf}
\IEEEoverridecommandlockouts    
\usepackage[latin1]{inputenc}
\usepackage{amsmath}
\usepackage[pdftex]{graphicx}
\usepackage[usenames]{color}
\usepackage{flushend}
\usepackage{setspace}
\usepackage{caption}
\usepackage{subcaption}
\usepackage{comment}
\usepackage{tikz}
\usepackage{bm}
\usepackage{xspace}
\usepackage{soul}
\usepackage{amsopn}
\usepackage{algpseudocode}
\usepackage{algorithm} 
\usepackage{enumerate}
\makeatletter
\let\NAT@parse\undefined
\makeatother
\usepackage{amsmath,amsfonts, amssymb,amsthm,color}
\usepackage[numbers,sort&compress]{natbib}
\usepackage{amsf onts}
\usepackage{color}
\usepackage{setspace}

\usepackage{amssymb}
\newtheorem{theorem}{Theorem}

\newtheorem{example}{Example}
\newtheorem{assumption}{Assumption}

\newtheorem{corollary}{Corollary}
\newtheorem{lemma}{Lemma}

\usepackage{breqn}
\usepackage{hyperref} 
\usepackage{epstopdf}
\usepackage{lipsum}
\usepackage{epsfig}
\DeclareGraphicsExtensions{.pdf,.eps,.png,.jpg,.mps}

\title{\LARGE \bf 
Q-linear Convergence of Distributed Optimization with Barzilai-Borwein Step Sizes.
}
\author{Iyanuoluwa Emiola 
\thanks{The author is with the Electrical \& Computer Engineering Department at the University of Central Florida, Orlando FL 32816, USA. Email: \texttt{iemiola@knights.ucf.edu }}%
}
\begin{document}
\maketitle


\begin{abstract}

The growth in sizes of large-scale systems and data in machine learning have made distributed optimization a naturally appealing technique to solve decision problems in different contexts. In such methods, each agent iteratively carries out computations on its local objective using information received from its neighbors, and shares relevant information with neighboring agents. Though gradient-based methods are widely used because of their simplicity, they are known to have slow convergence rates. On the other hand, though Newton-type methods have better convergence properties, they are not as applicable because of the enormous computation and memory requirements. In this work, we introduce a distributed quasi-Newton method with Barzilai-Borwein step-sizes. 
We prove a Q-linear convergence to the optimal solution, present conditions under which the algorithm is superlinearly convergent and validate our results via numerical simulations.
 \end{abstract}
 
 \section{Introduction}
Solving distributed optimization problems have different applications in different areas including distributed learning \cite{zeng2017energy,arjevani2015communication}. These problems are formulated in such a manner that agents in a network have to coordinate with other agents connected to them in a network to achieve a desired goal. In the distributed optimization problem below: 
\begin{equation}\label{eqn:original}
\min_x f(x) = \sum_{i=1}^n f_i(x),
\end{equation}
where $n$ represents the total number of agents in the network and the objective $f_i(\cdot)$ is only known by agent $i$. 

Typically, solutions to optimization or decision making problems over a network are approached using a gradient descent methods in dual domain or subgradient descent methods in the case of non-differentiable cost functions \cite{magnusson2016practical,lei2016primal,nedic2009distributed}. 
A major advantage of first-order methods, is the fact that their structure easily enables local computation for decision making in a distributed manner, where agents do not need global information to solve their decision problems via local computation. This simplicity of the computations associated with gradient methods make them amenable to distribute and parallelize in large-scale, computation-intensive problems such as machine learning \cite{recht2013parallel,lian2015asynchronous}. 
However, applying these gradient-based algorithms to large-scale problems face several challenges and become impractical due to their well-known slow convergence rates \cite{goffin1977convergence,schmidt2011convergence}.
To address the slow convergence rates of first order methods, second-order (Newton-type) methods have been proposed \cite{schraudolph2007stochastic,izmailov2004newton}. Even though Newton-type methods result in quadratic convergence, they also have a significant computational overhead from the need to invert and store the Hessian of the cost function. This computational burden makes them not suitable for large-scale systems and difficult to distribute, despite the efforts at adapting them for distributed implementation \cite{attouch2013convergence,zheng2014convergence}.

Quasi-Newton methods have been introduced as a way to leverage the fast convergence properties of second-order methods using the computation architecture of first-order methods. Quasi-Newton methods propose ways of incorporating the curvature information of the objective function from the second-order methods into the first-order approaches. Some examples include the BFGS \cite{eisen2017decentralized}, alongside its variations such as the low-memory BFGS \cite{mokhtari2015global}, the Barzilai-Borwein (BB) \cite{dai2005projected} and the DFP (Davidon - Fletcher - Powell method) \cite{pu2002convergence}, with different assumptions made on the cost function.

\subsection{Contributions} 
In this work, we present
     a fully distributed algorithm for solving an unconstrained optimization problem using uncoordinated BB step-sizes and obtain \textit{Q-linear} convergence when the cost function in Problem \eqref{eqn:original} is strongly convex. Furthermore, we show that for a class of objective functions, superlinear rate of convergence can be obtained. The results in this work are related to those in \cite{nedic2017achieving,nedic2017geometrically,dai2002r,dai2013new}.
     

The remainder of this section presents an overview of existing literature on convergence rates in distributed optimization and key results on convergence rates for distributed optimization using quasi-Newton methods. 
In Section \ref{sec:problemformulation}, the problem considered is presented and motivates the BB algorithm while Section \ref{sec:Convanalysiscentralized-BB} summarizes convergence results for the centralized BB case. Section \ref{sec:Dec-BB} presents the Distributed BB problem and algorithm and Section \ref{sec:convanadistributed} presents the convergence analysis of the D-BB algorithm -- the main results of this paper. The obtained results are illustrated and compared with other methods via numerical simulations in Section \ref{sec:Numerical}.

 \subsection{Literature Review}
Many decision problems in large-scale systems fit the paradigm for distributed optimization in which components or agents of the system locally iteratively solve  a part of an optimization problem and  exchange information with other (neighboring) agents in the network to arrive at a system-wide solution. Distributed decision problems are common in many areas including power systems, multi-agent systems, wireless sensor networks and have seen a recent surge in distributed machine learning, where server and worker nodes cooperate to solve learning problems as seen in \cite{li2014scaling}. 
Typically, in a network comprising $n$ agents, each agent has a (sometimes private) local objective function $f_i(x)$ and the goal is to optimize an aggregate function comprising the local objective functions of the agents $\sum_{i=1}^n f_i(x)$. In much of the literature, the local objective functions are usually assumed to be strongly convex and Lipschitz smooth.
 
 The literature on distributed optimization methods is rich and encompasses a wide range of methods that have been proposed to solve such problems including Alternating Direction Method of Multipliers (ADMM) dual averaging, gradient-based and Newton-type methods. While the ADMM framework can easily handle a broader class of functions e.g. nondifferentiable functions, can be parallelized and is easy to implement, it has a very poor convergence rate \cite{boyd2011distributed}. Dual averaging methods, on the other hand, in which agents keep estimates of the decision variable and exchange them with neighboring agents are known to perform even more poorly than ADMM \cite{agarwal2010distributed,duchi2011dual}.
 
Distributed gradient descent (DGD) methods and their variants, which attempt to combine the merits of the dual averaging and ADMM  are appealing and have been studied in the literature \cite{boyd2006randomized,nedich2001distributed,mateos2015distributed,nedic2008distributed,nedic2009distributed,assran2019stochastic,das2016distributed,berahas2018balancing}.
Applying the gradient-descent method to minimize strongly convex and smooth objective functions, is known, results in a linear rate of convergence with an appropriately chosen step size \cite{nesterov1998introductory}. 
To overcome the convergence rate limitations of gradient-based methods, second-order (Newton-type) algorithms have been proposed \cite{fischer1992special,polyak2007newton}. Though Newton-type methods have quadratic convergence rates, the computational and storage overhead incurred in inverting the Hessian is significant, particularly for large-scale problems that have a high number of variables especially in large distributed systems with multiple agents. Furthermore, to distribute the computation of Newton-type methods, positive-definiteness of the Hessian of the objective function is needed to ensure methods like matrix splitting are applicable \cite{jadbabaie2009distributed,wei2013distributed}. 

Even though constant or decaying step-sizes are commonly used in gradient-based and Newton-type methods, the \textit{Quasi-Newton} methods that leverage the computation structure of the gradient methods alongside the fast convergence properties of Newton-type methods have been studied, including methods like the Broyden-Fletcher-Goldfarb-Shanno (BFGS) algorithm \cite{eisen2017decentralized}, the Davidon-Fletcher-Powell \cite{pu2003revised} algorithm and the focus of this paper - the Barzilai-Borwein (BB) method first introduced in \cite{barzilai1988two}.  The central idea in the performance of these methods is to speed up convergence by exploiting the information from the inverse hessian without necessarily computing it explicitly. For example, Barzilai-Borwein quasi-Newton method computes step-sizes using the difference of successive iterates and the gradient evaluated at those iterates. One of the appealing properties of the BB method is the simple nature of the computations and updates involved, even in the distributed case as shown in this paper. 
 
This paper builds on earlier work on Distributed Gradient Descent (DGD) methods \cite{berahas2018balancing,jakovetic2014linear},   \cite{xin2018linear,nedic2017achieving,nedic2017geometrically}   as well as distributed Barzilai-Borwein methods \cite{dai2002r,dai2013new} where the authors analyze two-dimensional convex-quadratic functions. More recent efforts in \cite{gao2019geometric}, which took an adapt-then-combine strategy for agreement updates obtained a geometric rate of convergence.

In this paper, we propose a fully distributed algorithm that converges $Q$-linearly to the optimal solution. Furthermore, we show that a particular class of objective functions admit superlinear convergence with the modified Barzilai-Borwein algorithm presented in the paper.
The approach taken in this paper is applicable to strongly convex functions and we analyze the centralized and distributed cases where computation of the step-sizes are done in an uncoordinated manner.  In our approach, agents locally carry out computations, exchange information with neighboring agents to reach an agreement and use information obtained from other agents to continue the iterative process. 
Most of the notations used in this paper follow the norm in the literature and are briefly summarized next. 

\subsection{Notation}
Vectors and matrices are represented by boldface lower and upper case letters, respectively. We denote the set of positive and negative reals as $\mathbb{R}_+$ and $\mathbb{R}_-$, a vector or matrix transpose as $(\cdot)^T$, and the L$2$-norm of a vector by $||\cdot||$. The gradient of a function $f(\cdot)$ is denoted $\nabla f(\cdot)$.

\section{Problem formulation}\label{sec:problemformulation}

We consider the problem of the form below over a network of agents where their objective is to
\begin{equation}\label{eqn:quadratic}
\min_x f(x) = \sum_{i=1}^n f_i(x),
\end{equation}
In Problem \eqref{eqn:quadratic}, $f$ is strongly convex and smooth. Each agent $i$ in the network has access to a $f_i$ a component of $f$ and the agents collectively seek to optimize $f(x)$ by locally optimizing $f_i(x)$ iteratively. 

The communication graph of the multi-agent network is represented by an undirected weighted Graph $G = (\mathcal{V},\mathcal{E})$ in which $\mathcal{V} = {1,2, ... n}$ is the set of nodes (agents) and  $\mathcal{E}={(i,j)}$ is the set of edges such that agents $i,j$ are connected in the edge set, where $j\neq i$. The neighbors of agent $i$ is represented by the set $N_i = \{j: \ (i,j) \in \mathcal{E}\}$. Symmetry of the underlying graph implies that agents $i$ and $j$ for which $(i,j) \in \mathcal{E}$ means that information flows in both directions between both agents. 

The common approach to solve Problem \eqref{eqn:quadratic} is to use first-order methods, which involves updating the variable $x(k)$ iteratively  using the gradient of the cost function with the following equation:
\begin{equation}\label{eqn:gradient}
x(k+1) = x(k) - \alpha \nabla f(x(k)).
\end{equation}
It is well known that with an appropriate choice of the step size $\alpha$, the sequence $\{x(k)\}$ generated from Equation \eqref{eqn:gradient} converges to $x^*$.
%

The Newton method, which leverages curvature information of the cost function in addition to direction; and are known to speed up the convergence in the neighborhood of the optimal solution. The Newton-type methods have an update of the form:
\begin{equation}\label{eqn:newton}
x(k+1) = x(k) - \nabla f(x)(\nabla^2 f(x))^{-1}.
\end{equation}
Though they have good convergence properties, there are computational costs associated with building and computing the inverse hessian. In addition, some modification are needed if the hessian is not positive definite \cite{gill1972quasi}. 

\subsection{Barzilai-Borwein Quasi-Newton Method}\label{sec:about-BB}
The Barzilai-Borwein method differs from other quasi-Newton methods because it only uses one step size for the iteration as opposed to other quasi-Newton method that need approximations for the inverse of the hessian, thus, increasing the computation overhead. Problem \eqref{eqn:quadratic} is solved using the iterative scheme summarized in Algorithm \ref{alg:algorithm}, where a step-size $\alpha(k)$ is computed in the gradient descent method \eqref{eqn:gradient} so that $\alpha(k) \nabla f(x(k))$ approximates the $(\nabla^2 f(x(k)))^{-1} \nabla f(x(k))$ term in the Newton update \eqref{eqn:newton}.

Let
\[
s(k-1)\triangleq x(k)-x(k-1), \quad \text {and}
\]
\begin{equation}\label{eqn:8}
    y(k-1)=\nabla f(x(k))-\nabla f(x(k-1)). \nonumber
\end{equation}
The first BB step size is given by:
\begin{equation}\label{eqn:bb-step}
    \alpha_1(k)=\frac{s(k-1)^{T} s(k-1)}{s(k-1)^T y(k-1)}.
\end{equation}
Similarly, the second step size, $\alpha_2(k)$ is given by:
\begin{equation}\label{eqn:bb-step1}
     \alpha_2(k)=\frac{s(k-1)^{T} y(k-1)}{y(k-1)^T y(k-1)}.
\end{equation}
In general, there is flexibility in the choice to use $\alpha_1(k)$ or $\alpha_2(k)$ \cite{barzilai1988two}. In addition, both step sizes can be alternated within the same algorithm after a considerable amount of iterations to facilitate convergence. The procedure is summarized in Algorithm \ref{alg:algorithm}. 
\begin{algorithm}[h]
\caption{Algorithm for Centralized BB}
\label{alg:algorithm}
\begin{algorithmic}[1]
\Statex \textbf{Initialize}: $\alpha_1(0), x(0), \nabla f(x(0)), \varepsilon$.
\While{$\|\nabla f(x(k))\| \geq \varepsilon$}
\State Compute 
\[
\alpha_1(k)=\frac{s(k-1))^{T} s(k-1)}{s(k-1)^T y(k-1)}
\]{\small \Comment{$\alpha_2$ in Equation \eqref{eqn:bb-step1} may also be used}}
\State Update $x(k+1)=x(k)-\alpha_1(k)\nabla f(x(k))$\label{iterativecentralized1}
\EndWhile
\end{algorithmic}
\end{algorithm}

Before proceeding with the distributed BB algorithm and its convergence analyses, we first present a convergence analysis of the centralized case, where the following assumptions are made about Problem \eqref{eqn:quadratic} and Algorithm \ref{alg:algorithm}.
\begin{assumption}\label{assume-bounded-G}
The decision set $\mathcal{X}$ is bounded. This implies that there exists some constant $0 \leq B < \infty$ such that $|\mathcal{X}|\leq B$.
\end{assumption}
\begin{assumption}\label{Assumption1}
The objective function $f(x)$ in Problem \eqref{eqn:quadratic} is strongly convex and twice differentiable.
This implies that for $x,y \in \mathbb{R}^{np}$, there exists $\mu>0$ such that:
\begin{equation*}
    f_i(x)\geq f_i(y)+\nabla f_i(y)^{T}(x-y)+\frac{\mu}{2}\|x-y\|^2.
\end{equation*}
\end{assumption}
\begin{assumption}\label{Assumption3}
The inner product between the iterates deviations, $s$, and the gradient deviations, $y$, is strictly positive for all time step $k$. We make this assumption in both the centralized and distributed case.
\end{assumption}
\begin{assumption}\label{Assumption4}
  The gradient of the objective function $\nabla f$ is Lipschitz continuous. This implies that for all $x$ and $y$, there exists $L>0$ such that:
\begin{equation*}
   \|\nabla f(x)-\nabla f(y)\|\leq L\|x-y\|.
\end{equation*}
\end{assumption}

The assumptions above are typical in the literature for distributed optimization problems; and in fact, Assumption \ref{Assumption3} was made in \cite{gao2019geometric} as well.

\section{Convergence Analysis of Centralized BB}\label{sec:Convanalysiscentralized-BB}
In this section, we present a convergence analysis using the two BB step sizes for Problem \eqref{eqn:quadratic} where strong convexity of the cost function is assumed . 
\subsection{Convergence Analysis with First Step Size} \label{sec:convanalysisfirststepsize}
\begin{lemma}\label{lemmacentralized1}
Consider Algorithm \ref{alg:algorithm} for Problem \eqref{eqn:quadratic} and let Assumptions \ref{assume-bounded-G}, \ref{Assumption1},  \ref{Assumption3} and \ref{Assumption4} hold. In addition, let the Lipschitz continuity constant for $\nabla f(\cdot)$ and strong convexity parameter for $f(\cdot)$ satisfy $\mu\leq L$. If $\alpha_1(k)$ in Equation \eqref{eqn:bb-step} is such that $1/L \leq \alpha_1(k) \leq 2/(\mu + L)$, the iterates generated from Algorithm \ref{alg:algorithm} converge Q-Linearly to the optimal point $x^*$.
\end{lemma}
%
%
\begin{proof}

From Equation (\ref{eqn:gradient}), we first consider $\|x(k+1)-x^*\|^2$ to obtain bounds for convergence. First, we let $g(k)=\nabla f(x(k))$ and we obtain that:
\begin{equation*}
  \|x(k+1)-x^*\|=\|x(k)-x^*-\alpha_1(k)g(k)\|.
\end{equation*}
By squaring both sides, we have:
\begin{equation}
\begin{aligned}
\|x(k) {-} x^*{-}\alpha_1(k)g(k)\|^2 ={}  &\|x(k) {-} x^*\|^2 {+} \alpha_1^2(k)\|g(k)\|^2, \\
 & -2\left(x(k){-}x^*\right)^T\left(\alpha_1(k) g(k)\right). 
\end{aligned}
\end{equation}
Using the fact that for all vectors $a,\ b$, $2a^{T}b\leq \|a\|^2+\|b\|^2$.
So we obtain the relationship:
\begin{equation*}
2\left(x(k) - x^*\right)^T\left(g(k)\right) \leq \|g(k)\|^2 + \|x(k) - x^*\|^2.
\end{equation*}
By using strong convexity, where $\mu$ and $L$ are strong convexity and Lipschitz parameters respectively and $c_1$, $c_2$ are given by $c_1=2/(\mu +L)$ and $c_2=(2\mu L)/(\mu +L)$, we obtain:
\begin{equation}
\begin{aligned}
\|x(k) {-} x^*-\alpha_1(k)g(k)\|^2 \leq{}  &\|x(k) {-} x^*\|^2 {+} \alpha_1^2(k)\|g(k)\|^2 \\
 & -\alpha_1(k)c_1\|g(k)\|^2 \\
 & - \alpha_1(k)c_2\|x(k) {-} x^*\|^2, \\
 \leq  &(1-\alpha_1(k)c_2)\|x(k) {-} x^*\|^2\\ 
 &+(\alpha_1^2(k){-}\alpha_1(k)c_1)\|g(k)\|^2,\\
 \leq  &(1-\alpha_1(k)c_2)\|x(k) {-} x^*\|^2.
 \label{qlinearrate}
\end{aligned}
\end{equation}
and the last inequality is due to Theorem $2.1.12$ from chapter $2$ of \cite{nesterov1998introductory}. We also note that in the previous inequality, the term $(\alpha_1^2(k)-\alpha_1(k)c_2)\|g(k)\|^2\leq 0$ provided $\alpha_1(k)\leq c_1$.
%
We establish that the step size $\alpha_1(k)=c_1$ indeed is within the range of the BB step size bounds; that is, 
\begin{equation*}
    \frac{1}{L} \leq c_1 \leq \frac{1}{\mu},
\end{equation*}
and refer readers to Appendix \ref{Proof of Cor No11} for details.

Therefore the BB convergence can be analysed as:
\begin{eqnarray}
\|x(k+1) - x^*\|^2 &\leq& \left(1-\alpha_1(k)c_2\right)\|x(k) - x^*\|^2,\nonumber\\
 \dfrac{\|x(k+1) - x^*\|^2}{\|x(k) - x^*\|^2} &\leq& 1-\alpha_1(k)c_2,\nonumber\\
\text{and} \quad \dfrac{\|x(k+1) - x^*\|}{\|x(k) - x^*\|} &\leq& \left(1-\alpha_1(k)c_2\right)^{\frac{1}{2}}. \nonumber
\end{eqnarray}
 We will now analyse the right hand side of the above equation by bounding $\left(1-\alpha_1(k)c_2\right)^{\frac{1}{2}}$. The first Barzilai-Borwein step size $\alpha_{1}(k)$ is given by:
\begin{eqnarray}
\alpha(k) = \dfrac{\|s(k-1)\|^2}{\left[x(k) - x(k-1)\right]^T\left[\nabla f\left(x(k)\right) - \nabla f\left(x(k-1)\right)\right]}.\nonumber
\end{eqnarray}
By using Lipschitz continuity of $\nabla f(\cdot)$, with $L$ as the Lipschitz constant, we obtain the lower bound of the first BB step size as:
\begin{eqnarray}
\alpha_1(k) > \dfrac{\|x(k) - x(k-1)\|^2}{L\|x(k) - x(k-1)\|^2} = \dfrac{1}{L}.\nonumber
\end{eqnarray}
If $\alpha_1(k)$ and $c_2$ are positive and $\alpha_1(k) > 1/L$, then
 $ -\alpha_1(k)c_2 < -c_2/L$.
 Since $\alpha_1(k) > 1/L$,
it implies that $ 0 < 1-\alpha_1(k)c_2 < 1-c_2/L$. So we obtain the bound:
\begin{eqnarray}
 0 < \left(1-\alpha_1(k)c_2\right)^{\frac{1}{2}} < \left(1-\dfrac{c_2}{L}\right)^{\frac{1}{2}}.\nonumber
\end{eqnarray}
We will now show that $c_2/L<1$. If $c_2=2\mu L/(\mu +L)$, then it implies that $c_2/L=2\mu/(\mu +L)$.
If $\mu < L$, then it implies that $\mu +\mu < L+\mu$ and we obtain the fact that $2\mu/(\mu +L)<1$.

Therefore we obtain the relationship: $$\lim_{k\to\infty}\quad \dfrac{\|x(k+1) - x^*\|}{\|x(k) - x^*\|} <\left(1-\dfrac{c_2}{L}\right)^{\frac{1}{2}}< 1,$$
from which we conclude that the iterates $x(k)$ converge $Q$-linearly to the optimal point, $x^*$.
\end{proof}
\subsection{Convergence Analysis of Centralized BB with Second Step Size}\label{sec:convanalysissecondstepsize}
Here, we state a similar result to Lemma \ref{lemmacentralized1} using the second BB step size in Equation \eqref{eqn:bb-step1}.
\begin{lemma}\label{lemmacentralized2}
Consider Algorithm \ref{alg:algorithm} for Problem \eqref{eqn:quadratic} and let Assumptions \ref{assume-bounded-G}, \ref{Assumption1},  \ref{Assumption3} and \ref{Assumption4} hold. In addition, let the Lipschitz continuity constant for $\nabla f(\cdot)$ and strong convexity parameter for $f(\cdot)$ satisfy $\mu\leq L$. If $\alpha_2(k)$ in Equation \eqref{eqn:bb-step1} is such that $1/L \leq \alpha_2(k) \leq 2/(\mu + L)$, the iterates generated from Algorithm \ref{alg:algorithm} converge Q-Linearly to the optimal point $x^*$.
\end{lemma}
\begin{proof}
See Appendix \ref{Proof of Lemma No2}
\end{proof}

\section{Distributed Barzilai-Borwein Quasi-Newton Method}\label{sec:Dec-BB}

In this section we present a distributed solution to Problem \eqref{eqn:quadratic}, where Assumptions \ref{assume-bounded-G}, \ref{Assumption1}, \ref{Assumption3} and \ref{Assumption4} hold. In our proposed distributed algorithm, each agent in the network keeps a local copy of the decision variable $x_i(k)$ and a local gradient $\nabla f_i (x_i(k))$ and updates them at each time-step using locally computed step sizes $\alpha_i(k)$. 
The step size computation is similar to the centralized case. Using the local variables $x_i(k)$ and local gradient variables $\nabla f_i(x_i(k))$,  each agent computes
\begin{align}
    s_{i}(k-1) &= x_{i}(k)-x_{i}(k-1), \label{decentralizedupdatebeta}\\
        y_{i}(k-1) &=\nabla f_{i}(x_i(k))-\nabla f_{i}(x_i(k-1)), \label{decentralizedgradientbeta}
\end{align}
%
and computes $\alpha_i(k)$ in a manner that ensures
\begin{equation}\label{decentralizedstepsizecondition}
    (\alpha_i(k)^{-1}I)s_{i}{(k-1}) \approx y_{i}(k-1).
\end{equation}
%
%
Using the expressions in \eqref{decentralizedupdatebeta} and \eqref{decentralizedgradientbeta}, we obtain the distributed form of the step size for each agent $i$, which is given by:
\begin{equation}\label{decentralizedstepsize1}
    \alpha_{i1}(k)=\frac{(s_{i}(k-1))^{T}s_{i}(k-1)}{(s_{i}(k-1))^{T}y_{i}(k-1)}.
\end{equation}

and 
\begin{equation}\label{decentralizedstepsize2}
    \alpha_{i2}(k)=\frac{(s_{i}(k-1))^{T}y_{i}(k-1)}{(y_{i}(k-1))^{T}y_{i}(k-1)}.
\end{equation}
To distribute the computations locally at each time step $k$, each agent $i$ uses the following update scheme  $x_i \in \mathbb{R}^n$:
\begin{equation}\label{decentralizediteration}
    x_i(k+1)=x_i(k)-\alpha_i(k)\nabla f_i(x_i(k)).
\end{equation}
To ensure all agents converge to the optimal solution, each agent carries an iterative local computation step and the interaction with neighbors lead to a consensus step. Each agents takes a weighted average of the information received from its neighbors, to compute its next update. With this protocol, the local update at each agent $i$ is given by:
\begin{equation}\label{localoptimizationstep}
    x_i(k+1)=\sum_{j\in N_i\cup {i}} (w_{ij}x_j(k)-\alpha_i(k)\nabla f_i(x_i(k))).
\end{equation}
where $w_{ij}$ are weights attached by agent $i$ to agent $j$'s estimate.\\
Given that $W=[w_{ij}]$, and if we let $X=[x_1, \hdots, x_n] \in \mathbb{R}^{np}$ be the concatenation of local variables $x_i$, $I_p$ is the identity matrix of dimension $p$,  $\otimes$ represents the Kronecker operation of matrix product and $W=[w_{ij}]$ is the doubly stochastic weight matrix that satisfies: $W\otimes I_p \in \mathbb{R}^{np\times np}$. We can re-write Equation \eqref{localoptimizationstep} as follows:
\begin{equation}\label{concatenation}
    X(k+1)=(W \otimes I_p)X(k)-\alpha_i \nabla f(X(k)).
\end{equation}
where $X$ is the concatenation of local $x_i$, and $\nabla f(X(k))\in \mathbb{R}^{np}$ is the concatenated gradients. Similarly we can denote the average of local estimates to be $\overline{x}(k)$, the average of local estimates to be $g(x(k))$, the average of the Lipschitz constants for agents to be $L$, and the average of the strong convexity parameters for agents to be $\mu$. We also denote the averages of the two step sizes of agents ($\alpha_{i1}$ and $\alpha_{i1}$) to be $\overline{\alpha_{i1}}$ and $\overline{\alpha_{i2}}$ respectively. Because $W$ is doubly stochastic, it has one eigenvalue $\lambda=1$ and the other eigenvalues satisfy $0<\lambda<1$.



In the distributed implementation of the BB algorithm, neighbors compute their local step sizes using local information and exchange decision estimates with their neighbors over the communication network. The process is summarized in Algorithm \ref{alg:algorithm-dist}.
\subsection{Algorithm for Distributed BB}
\begin{algorithm}
\caption{Algorithm for Distributed BB}
\label{alg:algorithm-dist}
\begin{algorithmic}[1]
\Statex \textbf{Initialize}: $\alpha_i(0), x_i(0), \nabla f_i(x_i(0))$ 
\While{$\|g(x(k))\| \geq \varepsilon$ }
\State Compute
\begin{align*}
& s_{i}(k-1) \ \text{using \eqref{decentralizedupdatebeta}},\\ 
& y_{i}(k-1) \ \text{using \eqref{decentralizedgradientbeta}},\\
& \alpha_i(k+1) \ \text{using \eqref{decentralizedstepsize1}}.
\end{align*}
\State Local update in equation \eqref{localoptimizationstep}
\State Communicate updates $x_i(k+1)$ with neighbors.
\EndWhile
\end{algorithmic}
\end{algorithm}

\section{Convergence Analysis of Distributed BB}\label{sec:convanadistributed}
We examine convergence of Algorithm \ref{alg:algorithm-dist} to the optimal point based on the local estimates. 
\subsection{Distributed BB Convergence Analysis with the First Step-Size}
We present the main result of this section in Theorem \ref{mainresult}; and later prove it via a series of Lemmas in the rest of the section.
\begin{theorem}\label{mainresult}
Consider Algorithm \ref{alg:algorithm-dist} for Problem \eqref{eqn:quadratic} and let Assumptions \ref{assume-bounded-G}, \ref{Assumption1},  \ref{Assumption3} and \ref{Assumption4} hold. In addition, let the Lipschitz continuity constant for $\nabla f_i(\cdot)$ and strong convexity parameter for $f_i(\cdot)$ satisfy $\mu_i\leq L_i$ for each agent $i$. If $\alpha_{i1}(k)$ in Equation \eqref{decentralizedstepsize1} is such that $1/L_i \leq \alpha_{i1}(k) \leq 2/(\mu_i + L_i)$, the iterates of each agent $i$ generated from Algorithm \ref{alg:algorithm} converge Q-Linearly to the optimal point $x^*$; that is
\begin{equation}\label{eq:main-inequality}
    \|x_i(k)-x^*\|\leq\|x_i(k)-\overline{x}(k)\|+\|\overline{x}(k)-x^*\|.
\end{equation}
and we obtain a Q-Linear convergence after the agents reach consensus on the average value. Moreover, each local estimate $x_i(k)$ converges to the neighborhood of the optimal solution, $x^*$ based on the two step sizes $\alpha_{i1}$ and $\alpha_{i2}$.
\end{theorem}
To prove the main result in Theorem \ref{mainresult} we will take a two-step approach. First, we upper bound the norm of the difference between the individual agent iterates and the average of the agents' iterates in Lemma \ref{lemmabounds1}. Next, we show that the average of the agents' iterates converges Q-linearly to the optimal solution in Lemma \ref{lemmaboundsneighborhood1}.

\begin{lemma}\label{lemmabounds1}
Consider Algorithm \ref{alg:algorithm-dist} with BB step size $\alpha_{i1}$ in Equation \eqref{decentralizedstepsize1} for Problem \ref{eqn:quadratic} and suppose Assumptions \ref{assume-bounded-G}, \ref{Assumption1}, \ref{Assumption3} and \ref{Assumption4} hold; and 
$G$ be the upper bound of the gradients, then the norm of the difference between each local agent's estimate and the average agents' estimate is bounded by:

\begin{equation*}
\|x_i(k) - \overline{x}(k)\| \leq G \left(\sum_{m=0}^{k-1}\alpha_{i1}^2(m)\right)^{\frac{1}{2}}\left(\sum_{m=0}^{k-1}\lambda^{2(k-1-m)}\right)^{\frac{1}{2}},
\end{equation*}
where
\[
\left(\sum_{m=0}^k \alpha_i^2(m) \right)^{\frac{1}{2}} \leq \frac{\sqrt{k}}{\mu},
\]

and
\[\left(\sum_{m=0}^{k-1}\lambda^{2(k-1-m)}\right)^{\frac{1}{2}} \leq \left(\dfrac{1}{1-\lambda^{2}}\right)^{\frac{1}{2}} = \dfrac{1}{\sqrt{1-\lambda^{2}}} \triangleq\text{Q}_3.\]
Moreover if  
\begin{equation*}
    \sum_{m{=}0}^{k{-}1} \alpha_{i1}^2(m)\leq \frac{1}{G^2 \sum_{m=0}^{k-1}\lambda^{2(k-1-m)}},
\end{equation*}
then each local agent's estimates converges Q-linearly to its average; that is $\|x_i(k) - \overline{x}(k)\| \leq 1.$

\end{lemma}
\begin{proof}
The proof is presented in Appendix \ref{Proof of Lemma 4}
\end{proof}
After obtaining the norm of the difference between local estimates and the consensus average estimates, we now examine the convergence attribute of the average estimates $\bar{x}(k)$ to the optimal solution $x^*$.
Before proceeding, we state an important lemma that leads to the convergence behavior of average estimates to the optimal point.
  
  \begin{lemma}\label{lemmaboundsneighborhood1}
Consider Algorithm \ref{alg:algorithm-dist} for Problem \eqref{eqn:quadratic} and let Assumptions \ref{assume-bounded-G}, \ref{Assumption1},  \ref{Assumption3} and \ref{Assumption4} hold. In addition, let the Lipschitz continuity constant for $\nabla f(\cdot)$ and strong convexity parameter for $f(\cdot)$ satisfy $\mu\leq L$. If the average of the first distributed BB step size $\overline{\alpha_{i1}}$ is such that $1/L \leq \overline{\alpha_{i1}} \leq 2/(\mu + L)$,
For finite values of $i$ and $k$, and bounded gradients, the consensus average estimate converges Q-Linearly to the optimal point. Moreover, the local agent estimates converge to the neighborhood of the optimal point, $x^*$.
\end{lemma}

\begin{proof}
Let $\overline{x}(k) = \frac{1}{n}\sum\limits_{i=1}^n x_i(k)$, $g(k) = \frac{1}{n}\sum\limits_{i=1}^n \nabla f_i\left(x_i(k)\right)$, and $\overline{\alpha_{i1}}(k)=\frac{1}{n}\sum\limits_{i=1}^n \alpha_{i1}(k)$.

From equation \ref{decentralizediteration}, we first consider $\|\overline{x}(k+1)-x^*\|^2$ to obtain bounds for convergence. First, we let $g(k)$ be the average of gradient at local estimates and we let $\overline{\alpha_{i1}}$ be the average of the agents step sizes corresponding to the average of the iterates.
\begin{equation*}
  \|\overline{x}(k+1)-x^*\|=\|\overline{x}(k)-x^*-\overline{\alpha_{i1}}g(k)\|.
\end{equation*}
By squaring both sides and evaluating the right hand side, we have:
\begin{equation}
\begin{aligned}
\|x(k) {-} x^* {-}\overline{\alpha_{i1}}(k)g(k)\|^2 =& \|x(k) {-} x^*\|^2 + \overline{\alpha_{i1}}^2(k)\|g(k)\|^2 \\
 & -2\left(x(k){-}x^*\right)^T\left(\overline{\alpha_{i1}}(k) g(k)\right).
\end{aligned}
\end{equation}
Using the fact that for all vectors $a,\ b, \ 2a^{T}b\leq \|a\|^2+\|b\|^2$, we obtain the relationship:
\begin{equation*}
2\left(\overline{x}(k)- x^*\right)^T\left(g(k)\right) \leq \|g(k)\|^2 + \|\overline{x}(k) - x^*\|^2.
\end{equation*}
Just as we did for the centralized case, $\mu$ and $L$ are strong convexity and Lipschitz parameters respectively and $c_1$, $c_2$ are given by $c_1=2/(\mu +L)$ and $c_2=2\mu L/(\mu +L)$. We now obtain:
\begin{equation}
\begin{aligned}
\|\overline{x}(k{+}1) {-} x^*{-}\overline{\alpha_{i1}}(k)g(k)\|^2 
 \leq & \|\overline{x}(k){-} x^*\|^2 {+} \overline{\alpha_{i1}}^2(k)\|g(k)\|^2 \\
 & -\overline{\alpha_{i1}}(k)c_1\|g(k)\|^2 \\&- \overline{\alpha_{i1}}(k)c_2\|\overline{x}(k) {-} x^*\|^2, \\
  \leq  & (1-\overline{\alpha_{i1}}(k)c_2)\|\overline{x}(k) {-} x^*\|^2\\ &+(\overline{\alpha_{i1}}^2(k)-\overline{\alpha_{i1}}(k)c_1)\|g(k)\|^2,\\
 \leq & (1-\overline{\alpha_{i1}}(k)c_2)\|\overline{x}(k) {-} x^*\|^2.
 \label{qdistributedrate}
\end{aligned}
\end{equation}
We note that the last inequality is due to Theorem $2.1.12$ from chapter $2$ of \cite{nesterov1998introductory}. We also note that in the previous inequality, the term that contains the distributed form of the step  size, $(\overline{\alpha_{i1}}^2(k)-\overline{\alpha_{i1}}(k)c_1)\|g(k)\|^2\leq 0$ provided $\overline{\alpha_{i1}}(k)\leq c_1$.
We show that the step size $\overline{\alpha_{i1}}(k)=c_1$ is within the range of the BB step size bounds below:
\begin{corollary}\label{cor:cbb1}
Let $L$ and $\mu$ be the Lipschitz and strong convexity parameters respectively with $\mu \leq L$. The range of the average of the distributed BB step size $\overline{\alpha_{i1}}(k)$ is given by:
\end{corollary}
\begin{equation}\label{cbb-bound}
\frac{1}{L}\leq \overline{\alpha_{i1}}(k) \leq \frac{2}{\mu +L}\leq \frac{1}{\mu}.
\end{equation}
where the condition $\overline{\alpha_{i1}}(k)\leq \frac{1}{\mu}$ is assumed.
See Appendix \ref{Proof of Cor No11} for details.

Therefore the distributed BB convergence using the first BB step size can be analysed as:
\begin{eqnarray}
\|\overline{x}(k+1) - x^*\|^2 &\leq& \left(1-\overline{\alpha_{i1}}(k)c_2\right)\|\overline{x}(k) - x^*\|^2,\nonumber\\
 \dfrac{\|\overline{x}(k+1) - x^*\|^2}{\|\overline{x}(k) - x^*\|^2} &\leq& 1-\overline{\alpha_{i1}}(k)c_2,\nonumber\\
\text{and} \quad \dfrac{\|\overline{x}(k+1) - x^*\|}{\|\overline{x}(k) - x^*\|} &\leq& \left(1-\overline{\alpha_{i1}}(k)c_2\right)^{\frac{1}{2}}. \nonumber
\end{eqnarray}
 We will now bound: $\left(1-\overline{\alpha_{i1}}(k)c_2\right)^{\frac{1}{2}}$. 
 The distributed form of the first Barzilai-Borwein step size $\alpha_{i1}(k)$ is given by:
\begin{eqnarray}
\alpha_{i1}(k) = \dfrac{\|s_{i}(k-1)\|^2}{\left[x_{i}(k) - x_{i}(k-1)\right]^T\left[\nabla f_{i}\left(x_{i}(k)\right) - \nabla f_{i}\left(x_{i}(k-1)\right)\right]}\nonumber
\end{eqnarray}
By using Lipschitz continuity of $\nabla f(\cdot)$ with $L$ as the Lipschitz constant, we obtain the lower bound of distributed form of the first BB step size as:
\begin{eqnarray}
\alpha_{i1}(k) > \dfrac{\|x_{i}(k) - x_{i}(k-1)\|^2}{L\|x_{i}(k) - x_{i}(k-1)\|^2} = \dfrac{1}{L}.\nonumber
\end{eqnarray}
We know that $\overline{\alpha_{i1}}(k)=\frac{1}{n}\sum\limits_{i=1}^n \alpha_{i1}(k)$, and it follows that $n\overline{\alpha_{i1}}(k)=\sum\limits_{i=1}^n \alpha_{i1}(k)$. But we know that $\alpha_{i1}(k)>\frac{1}{L}$, and as a fact, $\alpha_{i1}(k)<\sum\limits_{i=1}^n \alpha_{i1}(k)$. Therefore we obtain the relationship:
\begin{equation}\label{sum-step-size-bound}
    \frac{1}{L}<\alpha_{i1}(k)<\sum\limits_{i=1}^n \alpha_{i1}(k).
\end{equation}
From equation \eqref{sum-step-size-bound}, $n\overline{\alpha_{i1}}(k)=\sum\limits_{i=1}^n \alpha_{i1}(k)>\frac{1}{L}$ and we obtain the fact that $\overline{\alpha_{i1}}(k)>\frac{1}{nL}$.

If $\overline{\alpha_{i1}}(k)$ and $c_2$ are positive and $\overline{\alpha_{i1}}(k) > 1/nL$, then
 $ -\overline{\alpha_{i1}}(k)c_2 < -c_2/nL$.
  Since $\overline{\alpha_{i1}}(k) > 1/nL$.
it implies that $ 0 < 1-\overline{\alpha_{i1}}(k)c_2 < 1-c_2/nL$,
Therefore,
\begin{eqnarray}
 0 < \left(1-\overline{\alpha_{i1}}(k)c_2\right)^{\frac{1}{2}} < \left(1-\dfrac{c_2}{nL}\right)^{\frac{1}{2}}\nonumber
\end{eqnarray}
We will now show that $c_2/nL<1$. If $c_2=2\mu L/(\mu +L)$, then it implies that $c_2/nL=2\mu /n(\mu +L)$.
If $\mu \leq L$, then we have $\mu +\mu \leq L+\mu$ and we obtain that $2\mu/n(\mu +L)\leq1$ for all positive values of $n$.
Therefore, $$\lim_{k\to\infty}\quad \dfrac{\|\overline{x}(k+1) - x^*\|}{\|\overline{x}(k) - x^*\|} \leq\left(1-\dfrac{c_2}{nL}\right)^{\frac{1}{2}}\leq 1,$$
So the average of the estimates converges Q-linearly to the optimal point, $x^*$.
\end{proof}

\subsection{Distributed BB with Second Step-Size}
\begin{lemma}\label{lemmabounds2}
Suppose Assumptions  \ref{assume-bounded-G}, \ref{Assumption1}, \ref{Assumption3} and \ref{Assumption4} hold. For the second BB step size given by:
\begin{equation*}
    \alpha_{i2} = \frac{s_i(k-1)^Ty_i(k-1)}{y_i(k-1)^Ty_i(k-1)}.
\end{equation*}
For finite values of $i$ and $k$, and bounded gradients, where $G$ is the upper bound of the gradients, the norm of the difference between the local agents estimate and the consensus average estimate is bounded and given by:
\begin{eqnarray}
\|x_i(k) - \overline{x}(k)\| \leq G \left(\sum_{m=0}^{k-1}\alpha_{i2}^2(m)\right)^{\frac{1}{2}}\left(\sum_{m=0}^{k-1}\lambda^{2(k-1-m)}\right)^{\frac{1}{2}}.
\end{eqnarray}
Where:
\[
\left(\sum_{m=0}^k \alpha_i^2(m) \right)^{\frac{1}{2}} \leq \frac{\sqrt{k}}{\mu}.
\]
and:
\[\left(\sum_{m=0}^{k-1}\lambda^{2(k-1-m)}\right)^{\frac{1}{2}} \leq \left(\dfrac{1}{1-\lambda^{2}}\right)^{\frac{1}{2}} = \dfrac{1}{\sqrt{1-\lambda^{2}}} \triangleq \text{Q}_3\].
\end{lemma}

\begin{proof}
The proof of this result is similar to the proof of Lemma \ref{lemmabounds1}; see Appendix \ref{Proof of Lemma 6}
\end{proof}
Now, we bound the convergence of average of local estimates to the optimal point using the second step size. First let us consider $\|\overline{x}(k+1) - x^*\|^2$:
We will state a lemma before we proceed to the convergence analysis.
\begin{lemma}\label{lemmaboundsneighborhood2}
Consider Algorithm \ref{alg:algorithm-dist} for Problem \eqref{eqn:quadratic} and let Assumptions \ref{assume-bounded-G}, \ref{Assumption1},  \ref{Assumption3} and \ref{Assumption4} hold. In addition, let the Lipschitz continuity constant for $\nabla f(\cdot)$ and strong convexity parameter for $f(\cdot)$ satisfy $\mu\leq L$. If the average of the second distributed BB step size $\overline{\alpha_{i2}}$ is such that $1/L \leq \overline{\alpha_{i2}} \leq 2/(\mu + L)$,
For finite values of $i$ and $k$, and bounded gradients, the consensus average estimate converges Q-Linearly to the optimal point. Moreover, the local agent estimates converge to the neighborhood of the optimal point, $x^*$.
\end{lemma} 
\begin{proof}
See Appendix \ref{secondstepsizeconvproof}.
\end{proof}

\section{On Superlinear Convergence of Algorithm 1}\label{sec:Superlinear}
In the preceding sections, we obtained Q-linear convergence for Algorithms \ref{alg:algorithm} and \ref{alg:algorithm-dist}, when the cost function being minimized is strongly convex with Lipschitz-continuous gradients. A consequence of Lemmas \ref{lemmaboundsneighborhood1} and \ref{lemmaboundsneighborhood2} is that under certain conditions, Algorithm \ref{alg:algorithm-dist} converges superlinearly to the optimal solution $x^*$.
\begin{corollary}\label{cor:superlinear}
Consider Problem \ref{eqn:quadratic}. If the objective function $f(\cdot)$ being minimized is such that its strong convexity parameter is equal to the Lipschiz gradient, then the iterates generated by Algorithm \ref{alg:algorithm} converges superlinearly to the optimal solution.
\end{corollary}

\begin{example}
An example of a function that meets the conditions required for super linear convergence in Corollary \ref{cor:superlinear} is the quadratic function:
\begin{equation}\label{eqn:superlinear}
    f(x)=0.5x^{T}x.
\end{equation}
in a distributed manner.
\end{example}

Superlinear convergence can be obtained for strongly convex quadratic functions and two-dimensional strictly convex quadratic functions \cite{dai2013new}. However, we illustrate an example of a strongly convex quadratic function that yield a superlinear convergence behavior.  The strong convexity parameter of the cost function in  \eqref{eqn:superlinear} is equal to the Lipschitz parameter of its gradient; that is, $L=\mu =1$. By using this function on the upper bound obtained for rates of convergence as seen in Lemma \ref{lemmaboundsneighborhood1} and Lemma \ref{lemmaboundsneighborhood2}, we have:
\begin{equation*}
    \lim_{k\to\infty}\quad \dfrac{\|x(k+1) - x^*\|}{\|x(k) - x^*\|} <\left(1-\dfrac{c_2}{L}\right)^{\frac{1}{2}},
\end{equation*}
where $c_2=2\mu L/(\mu +L)$.
To obtain superlinear convergence rate, if $L=\mu$, then we obtain:
\begin{equation*}
    \left(1-\dfrac{c_2}{L}\right)=\left(1-\dfrac{2\mu}{2\mu}\right)=0.
\end{equation*}
and we obtain:
\begin{equation*}
\lim_{k\to\infty}\quad \dfrac{\|x(k+1) - x^*\|}{\|x(k) - x^*\|} < 0,
\end{equation*}
and we obtain a superlinear convergence rate.
\section{Numerical Experiments}\label{sec:Numerical}
We show some simulations for results in Lemma \ref{sec:Convanalysiscentralized-BB} (the centralized case) and the results in Theorem \ref{mainresult} (the distributed case). For the centralized, we consider the least square cost function:
\begin{equation}\label{eq:centralized-sim-obj}
    f(x) = \|Ax-b\|_2^{2} ,
\end{equation}
where the matrix $A$ is positive-definite so that $f(x)$ is strongly convex with parameter $\mu$ and $\mu \leq L$, where $L$ is the Lipschitz parameter of the gradient of $f(x)$. We run the simulations for $50$ iterations and compare the BB step size in \eqref{eqn:bb-step} with the gradient method using a decaying step size of $\frac{1}{k}$. In Figure \ref{fig:centralized-simulations}, the label "gradient" is the curve obtained when a step size of $\frac{1}{k}$ is used using a gradient descent algorithm while the label "BB" is the the curve obtained when the actual BB step size in equation \eqref{eqn:bb-step} is used. The results of the simulation of Algorithm \ref{alg:algorithm} are summarized in Figure \ref{fig:centralized-simulations} where the errors as a function of time are illustrated. As can be observed, the error converges to zero indicating the iterates $x(k)$ are approaching the optimal solution $x^*$.

For the distributed case, we consider the following objective function, which is separable per agent:
\begin{equation}\label{leastsquare}
    f(x)=\frac{1}{2}\sum_{i=1}^{n}x^{T}A_i x +b_i^{T}x,
\end{equation}
where $A_i \in \mathbb{R}^{p\times p}$, $b_i \in \mathbb{R}^{p}$ are used by each agent $i$ for their own computation, and $n$ is the number of agents in the network and its dimension is $m=10$. In equation \eqref{leastsquare}, the gradient function is given by:
\begin{equation}\label{leastsquare1}
    \nabla f(x)=\frac{1}{2}(A_i +A_i^{T})x +b_i.
\end{equation}
We note that the function \eqref{leastsquare} is strongly convex and its gradient function in \eqref{leastsquare1} is Lipschitz continuous for an appropriate value of $A$. We verify this through its strong convexity parameters $\mu_i$ and Lipschitz parameter values $L_i$, where $\mu_i \leq L_i$ for each agent $i$ in the network. We also note that $\mu_i$ is the maximum of all the eigenvalues of matrix $(A_i +A_i^{T})$ and $L_i$ is the spectral norm of matrix $(A_i +A_i^{T})$. We will use a scenario where there are $100$ nodes in the network and the matrix $W$ is a positive, symmetric, random doubly stochastic matrix. Our simulations aim to compare different step sizes with the distributed Barzilai-Bowein in equations \eqref{decentralizedstepsize1} and \eqref{decentralizedstepsize2}. Specifically, we use the following step sizes of $\alpha_i=\frac{1}{L_i}$,  $\alpha_i=\frac{1}{\mu_i}$, and $\alpha_i = \frac{2}{L_i + \mu_i}$ according to Lemmas \ref{lemmaboundsneighborhood1} and \ref{lemmaboundsneighborhood2}, and the BB step size as seen in equation \eqref{decentralizedstepsize1}. In Figure \ref{fig:distributedsimulations}, the label $c_1=\frac{2}{\mu +L}$ bound curve (the circular curve) is the step size according to convergence result in Lemma \ref{lemmaboundsneighborhood1}. The label BB step size in figure \ref{fig:distributedsimulations} (the curve beneath all other curves) is the actual BB step size in equation \eqref{decentralizedstepsize1} and the labels BB-Upperbound step size curve (the one in asterisk) and BB-Lowerbound step size curve ( triangular) are the lower and upper bounds of the BB step sizes ($\alpha_i=\frac{1}{L_i}$,  $\alpha_i=\frac{1}{\mu_i}$) respectively. Our simulations affirm that the step size $\alpha_i = \frac{2}{L_i + \mu_i}$ lies in between the lower and upper bounds of the BB step size and also agrees with the theoretical result in Corollary \ref{cor:cbb1}. We apply these step sizes to the iteration  shown in equation \eqref{localoptimizationstep} to compare the rates at which each step size converges to the optimal point. By setting the gradient function in equation \eqref{leastsquare1}, to zero, we obtain the optimal solution $x^*$ given by the following equation:
\begin{equation}\label{optimalsolutionequation}
    x^*=-2(A_i+A_i^{T})^{-1}b.
\end{equation}
The optimal solution in equation \eqref{optimalsolutionequation} is then obtained for the three different step sizes we used for simulation. Though convergence is attained for the three step sizes used, we run the simulations for $50$ iterations to compare convergence speeds. We compare convergence rates by first initializing $x$, $A$ and $b$ as zeros between time step $k=1$ to the total number of iterations $T=50$. We also initialize the values of $L$, $\mu$ and the three forms of the step sizes as zeros and run the simulation for $1000$ iterations. We plot the error curve for the three step sizes indicated by the expression $\|\overline{x}(k+1)-x^*\|$ and compare the results.
Our Numerical results are shown graphically below in Figures \ref{fig:centralized-simulations} and \ref{fig:distributedsimulations}:
\begin{figure}[h]
    \centering
    \includegraphics[width=8cm]{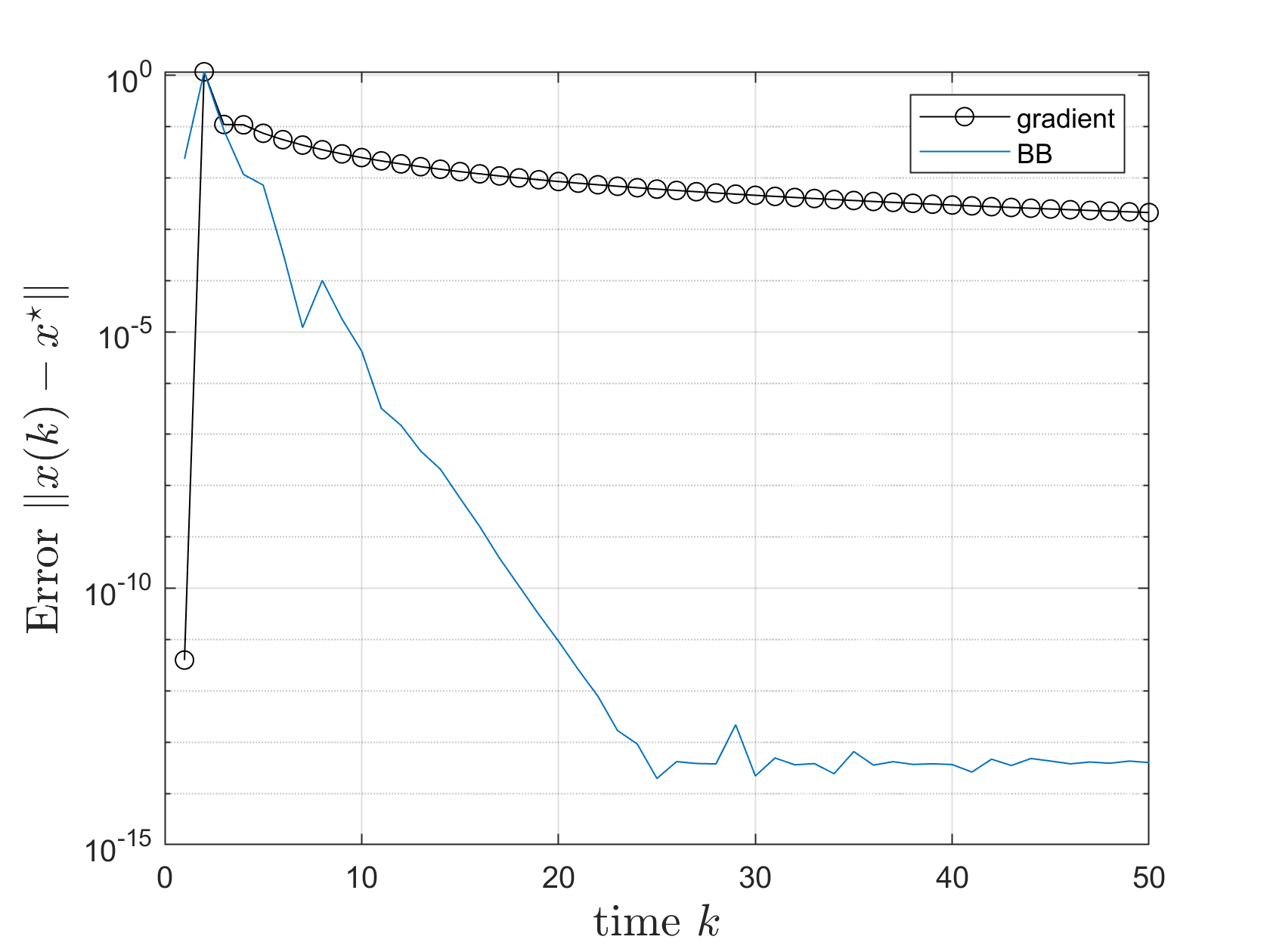}
    \caption{Centralized Simulations for $50$ iterations}
    \label{fig:centralized-simulations}
\end{figure}

\begin{figure}[h]\label{Simulations-Distributed}
    \centering
    \includegraphics[width=8cm]{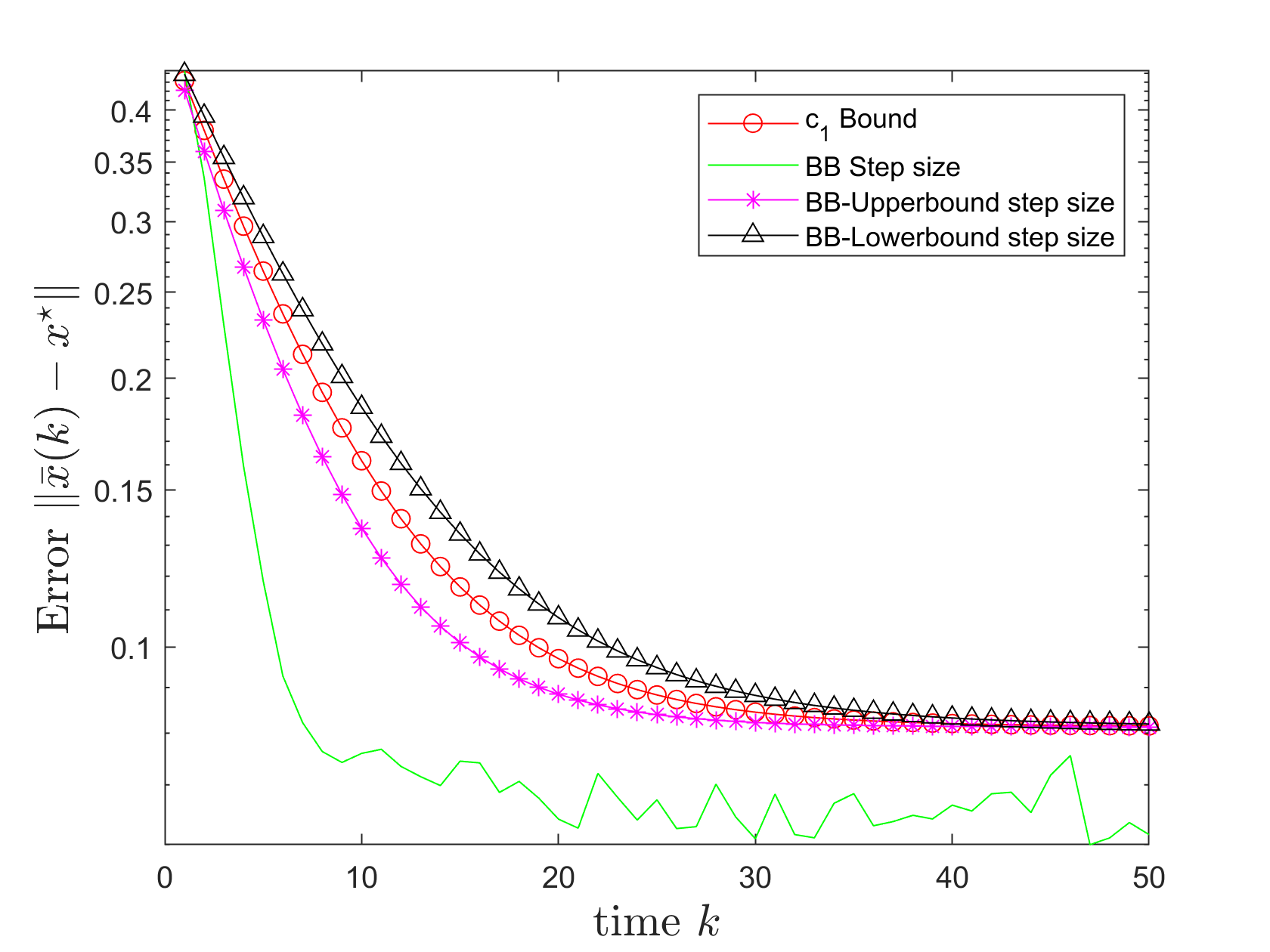}
    \caption{Distributed Simulations for $50$ iterations}
    \label{fig:distributedsimulations}
\end{figure}

\section{Conclusions}
\label{sec:conclude}
We examined the convergence attributes of an unconstrained problem and we obtained Q-linear convergence for both the centralized and distributed case by assuming strong convexity on the objective function. However, because this is a consensus-type scenario, the Q-linear convergence rate holds when the agents in a network locally agree to an average, $\overline{x}(k)$ for both the centralized and distributed cases. We also compared the speed of convergence for the gradient method with the BB method by using the two Barzilai-Borwein step sizes, and the BB method converges to the optimal point at a faster than the gradient method. In section \ref{sec:Numerical}, the faster convergence of the BB method is evident because of its minimal error in getting to the optimal point than in the gradient method after some iterations. 
  
  
\bibliographystyle{IEEEtran}
{\small
\bibliography{dbb-ref-1}}

\appendix
\subsection{Proof of Corollary \ref{cor:cbb1}}\label{Proof of Cor No11} 
\begin{proof}
We establish equation \eqref{cbb-bound} in a manner that the range of step size bounds below holds: 
\begin{equation*}
\frac{1}{L}\leq \alpha_{i1}(k) \leq \frac{2}{\mu +L}\leq \frac{1}{\mu}.
\end{equation*}
We also note that the step size range also applied to both the centralized and distributed form of the step sizes.
First, according to \cite{gao2019geometric}, we first start by noting that the BB step size can be upper and lower bounded according to:
\begin{equation*}
\frac{1}{L}\leq \alpha_{i1}(k) \leq \frac{1}{\mu}.
\end{equation*}
To include $2/(\mu +L)$  between the first distributed step size, $\alpha_{i1}(k)$ and $1/\mu$, we prove that $2/(\mu +L)\leq 1/\mu$ and $2/(\mu +L)\geq 1/L$. 

To prove that $2/(\mu +L)\leq 1/\mu$, we show that $\frac{1}{\mu}-\frac{2}{\mu +L}>0$.
\begin{equation*}
    \frac{1}{\mu}-\frac{2}{\mu +L}=\frac{L-\mu}{\mu (\mu +L)}.
\end{equation*}
We know that $L\geq \mu$ and $L$ and $\mu$ are positive, then we obtain that $1/\mu-2/(\mu +L)>0$.

Now we will prove that $2/(\mu +L)\geq 1/L$. below:
\begin{equation*}
    \frac{2}{\mu +L}-\frac{1}{L}=\frac{L-\mu}{L(L+\mu)}.
\end{equation*}
Since $L\geq \mu$ and $L$ and $\mu$ are both positive, we have that $2/(\mu +L)\geq 1/L$. Therefore we obtain the bounds:
\begin{equation*}
\frac{1}{L}\leq \alpha_{i1}(k) \leq \frac{2}{\mu +L}\leq \frac{1}{\mu}.
\end{equation*}
So we conclude that the step size $\alpha_{i1}(k)=2/(\mu +L)$ lies between the lower ($1/L $) and lower bounds ($1/\mu )$ of the BB step sizes. The fact that $\overline{\alpha_{i1}}(k)=\frac{1}{n}\sum\limits_{i=1}^n \alpha_{i1}(k)$ and $\overline{\alpha_{i1}}(k)\leq \frac{1}{\mu}$ hold completes the proof.
\end{proof}
\subsection{Proof of Lemma \ref{lemmacentralized2}}\label{Proof of Lemma No2}
\begin{proof}
From Equation (\ref{eqn:gradient}), we first consider $\|x(k+1)-x^*\|^2$ to obtain bounds for convergence. First, we let $g(k)=\nabla f(x(k))$ and we obtain that:
\begin{equation*}
  \|x(k+1)-x^*\|=\|x(k)-x^*-\alpha_2(k)g(k)\|.
\end{equation*}
By squaring both sides, we have:
\begin{equation}
\begin{aligned}
\|x(k) - x^*-\alpha_2(k)g(k)\|^2 ={}  &\|x(k) - x^*\|^2 + \alpha_2^2(k)\|g(k)\|^2, \\
 & -2\left(x(k){-}x^*\right)^T\left(\alpha_2(k) g(k)\right).
\end{aligned}
\end{equation}
For all vectors $a,\ b$, $2a^{T}b\leq \|a\|^2+\|b\|^2$.
We now have the relationship:
\begin{equation*}
2\left(x(k) {-} x^*\right)^T\left(g(k)\right) \leq \|g(k)\|^2 {+} \|x(k) {-} x^*\|^2.
\end{equation*}
By strong convexity assumption with parameter $\mu$, Lipschitz constant $L$, and constants $c_1$, $c_2$ expressed as $c_1=2/(\mu +L)$ and $c_2=2\mu L/(\mu +L)$, we obtain:
\begin{equation}
\begin{aligned}
\|x(k) {-} x^*{-}\alpha_2(k)g(k)\|^2 \leq{}  &\|x(k) {-} x^*\|^2 + \alpha_2^2(k)\|g(k)\|^2 \\
 & -\alpha_2(k)c_1\|g(k)\|^2 \\
 & - \alpha_2(k)c_2\|x(k) - x^*\|^2, \\
 \leq  &(1-\alpha_2(k)c_2)\|x(k) - x^*\|^2\\ 
 &+(\alpha_2^2(k)-\alpha_2(k)c_2)\|g(k)\|^2,\\
 \leq  &(1-\alpha_2(k)c_2)\|x(k) - x^*\|^2.
 \label{qlinearrate2}
\end{aligned}
\end{equation}
where the last inequality is due to Theorem $2.1.12$ from chapter $2$ of \cite{nesterov1998introductory} and the term $(\alpha_2^2(k)-\alpha_2(k)c_2)\|g(k)\|^2\leq 0$ provided $\alpha_2(k)\leq c_1$.
According to the verification as seen in a similar manner in the proof of corollary  \ref{Proof of Cor No11}, the second step size bounds satisfies:
\begin{equation*}
\frac{1}{L}\leq \alpha_2(k) \leq \frac{2}{\mu +L}\leq \frac{1}{\mu}.
\end{equation*}
Therefore we have:
\begin{eqnarray}
\|x(k+1) - x^*\|^2 &\leq& \left(1-\alpha_2(k)c_2\right)\|x(k) - x^*\|^2,\nonumber\\
 \dfrac{\|x(k+1) - x^*\|^2}{\|x(k) - x^*\|^2} &\leq& 1-\alpha_2(k)c_2,\nonumber\\
\text{and} \quad \dfrac{\|x(k+1) - x^*\|}{\|x(k) - x^*\|} &\leq& \left(1-\alpha_2(k)c_2\right)^{\frac{1}{2}}. \nonumber
\end{eqnarray}
Now we analyse the right hand side of the above equation by bounding $\left(1-\alpha_2(k)c_2\right)^{\frac{1}{2}}$. The second Barzilai-Borwein step size $\alpha_{2}(k)$ is given by:
\begin{eqnarray}
 \alpha_2(k)=\frac{s(k-1)^{T} y(k-1)}{y(k-1)^T y(k-1)}.
\end{eqnarray}
By using Lipschitz continuity of $\nabla f(\cdot)$, with $L$ as the Lipschitz constant, we obtain according to \cite{gao2019geometric} that the second BB step size is lower bounded by $1/L$.
If $\alpha_2(k)$ and $c_2$ are positive and $\alpha_2(k) > 1/L$, then
 $ -\alpha_2(k)c_2 < -c_2/L$.
 \noindent Since $\alpha_2(k) > 1/L$,
then 
$$ 0 < 1-\alpha_2(k)c_2 < 1-\dfrac{c_2}{L}, \quad \text{which implies that}$$
 \[ 
 0< \left(1-\alpha_2(k)c_2\right)^{\frac{1}{2}} < \left(1-\dfrac{c_2}{L}\right)^{\frac{1}{2}}.
\]
To prove that $c_2/L<1$, if $c_2=2\mu L/(\mu +L)$, then we obtain the fact that $c_2/L=2\mu /(\mu +L)$.
If $\mu < L$, it implies that $\mu +\mu < L+\mu$ and we obtain $2\mu/(\mu +L)<1$.
Therefore we have the following relationship: $$\lim_{k\to\infty}\quad \dfrac{\|x(k+1) - x^*\|}{\|x(k) - x^*\|} <\left(1-\dfrac{c_2}{L}\right)^{\frac{1}{2}}< 1,$$
from which we conclude that the iterates $x(k)$ converge $Q$-linearly to the optimal point, $x^*$.
\end{proof}
\subsection{Proof of Lemma \ref{lemmabounds1}}\label{Proof of Lemma 4}
\begin{proof}
We first consider the first step size $\alpha_{i1}$ as expressed in equation \eqref{decentralizedstepsize1}, and 
for notation simplicity, we denote $ Z = W \otimes I_p \in \mathbb{R}^{np\times np}$.
Then the distributed iteration at time step $k$ is given by:
\begin{equation}
X(k+1) =  ZX(k) - \alpha_{i1}(k)\nabla f(X(k)).
\label{distributedstepsize1}
\end{equation}
where $\otimes$ is the kronecker product.
From the definitions of $\overline{x}(k)$ as the average of local estimates, we have the expression:
$\overline{x}(k) = \frac{1}{n}\sum\limits_{i=1}^n x_i(k).$ Likewise, from the definition of $g(x(k))$ as the average of local gradients, we obtain the relationship:
 $g(x(k)) = \frac{1}{n}\sum\limits_{i=1}^n \nabla f_i\left(x_i(k)\right).$
If we solve for $X(k)$ in equation (\ref{distributedstepsize1})   where $\alpha_{i1}$ is the BB step size, we obtain the expression:
\begin{equation}\label{bigconcatenation}
X(k) = -\sum_{m=0}^{k-1} \alpha_1(m) \left(W^{(k-1-m)}\otimes I\right)\nabla f\left(x(m)\right).
\end{equation}
Suppose $\overline{X}(k)$ is the average of all concatenated $x_i(k)$, 
then we obtain
\begin{equation}
\label{averageequation}
\overline{X}(k) = \dfrac{1}{n}\left(\left(1_n 1_n^T\right) \otimes I\right)x(k).
\end{equation}
Also, we know that the following relationship:
\begin{equation}
\|x_i(k) - \overline{x}(k)\| \leq \|X(k) - \overline{X}(k)\|.
\label{smallandbig}
\end{equation}
where the distributed form of the first BB step size is given in $\alpha_{i1}$ as expressed in equation \eqref{decentralizedstepsize1}.
From equations \eqref{bigconcatenation} and \eqref{averageequation}, we obtain:
\begin{equation}
\begin{aligned}
\|x_i(k) - \overline{x}(k)\| 
 \leq & \|X(k) - \overline{X}(k)\|, \\
= & \|X(k) -\frac{1}{n}\left(\left(1_n1_n^T\right)\otimes I\right)X(k)\|,\\
 = & \|-\sum\limits_{m=0}^{k{-}1}\alpha_{i1}(m)(W^{(k{-}1{-}m)}\otimes I)\nabla f(x(m)) \\
  & +\frac{1}{n}((1_n1_n^T)\otimes I)\sum_{m{=}0}^{k{-}1}(\alpha_{i1}(m) \\
 & (W^{(k{-}1{-}m)}\otimes I)\nabla f(x(m)))\|, \\
 = & \|-\sum\limits_{m=0}^{k{-}1}(\alpha_{i1}(m)(W^{(k{-}1{-}m)}\otimes I)\nabla f(x(m))) \\ 
  &  +\sum\limits_{m=0}^{k-1}\frac{\alpha_{i1}(m)}{n} ((1_n1_n^T)\otimes I)\nabla f(x(m))\|.
\end{aligned}
\end{equation} 
Because $W$ is doubly stochastic, then we have the relationship:
\begin{equation*}
\begin{aligned}
\|x_i(k) - \overline{x}(k)\| 
\leq& -\sum_{m=0}^{k-1}\| \alpha_{i1}(m)((W^{(k-1-m)}-\frac{1}{n}1_n1_n^T)\\
& \otimes I)\nabla f(x(m)).\|\nonumber
\end{aligned}
\end{equation*}
If 
$\alpha_{i1}(k) = \frac{s_i(k-1)^Ts_i(k-1)}{s_i(k-1)^Ty_i(k-1)}$,
then
$\alpha_{i1}(m) = \frac{s_i(m-1)^Ts_i(m-1)}{s_i(m-1)^Ty_i(m-1)}$.
Therefore we have the expression:
\begin{eqnarray}
\|x_i(k) - \overline{x}(k)\| & \leq & -\sum_{m=0}^{k-1}\frac{\|s_i(m-1)\|^2}{s_i(m-1)^Ty_i(m-1)}.\nonumber\\
&& \|W^{(k-1-m)} -  \frac{1}{n}1_n1_n^T\|\|\nabla f(x(m))\|.\nonumber
\end{eqnarray}
So we obtain that the expression: \\$\|x_i(k) - \overline{x}(k)\|$.
\begin{eqnarray}
\leq  \sum_{m=0}^{k-1} \frac{\|s_i(m-1)\|^2}{s_i(m-1)^Ty_i(m-1)}\lambda^{(k-1-m)}\|\nabla f(x(m)).\|\nonumber
\end{eqnarray}
If $\nabla f(x(m))$ is bounded meaning that $\|\nabla f(x(m)\| \leq G$ where $G$ is positive, then we have:
\begin{equation}
\|x_i(k) - \overline{x}(k)\|\leq \sum_{m=0}^{k-1} \frac{G\|s_i(m-1)\|^2}{s_i(m-1)^Ty_i(m-1)}\lambda^{(k-1-m)}.
\label{distriteration}
\end{equation}
The eigenvalues $\lambda$ of the weight matrix $W$ satisfies the bounds, $0<\lambda\leq1$.
From equation (\ref{distriteration}) and by Cauchy-Schwarz on sums, we obtain:
\begin{eqnarray}
\|x_i(k){-} \overline{x}(k)\| \leq G\left(\sum_{m=0}^{k-1}\alpha_{i1}^2(m)\right)^{\frac{1}{2}}\left(\sum_{m=0}^{k-1}\lambda^{2(k-1-m)}\right)^{\frac{1}{2}}.
\label{disiteration2}
\end{eqnarray}
We know that the BB step size is upper bounded such that $\alpha_i < 1/\mu$. \\
Then by squaring both sides, $\alpha_i^2 \leq \frac{1}{\mu^2}$ and we obtain the relationship $\sum_{m=0}^k \alpha_i^2(m) \leq \frac{k}{\mu^2}$ and we obtain the result $\left(\sum_{m=0}^k \alpha_i^2(m) \right)^{\frac{1}{2}} \leq \frac{\sqrt{k}}{\mu}$.

\noindent In equation (\ref{disiteration2}),
\begin{eqnarray}
\|x_i(k) - \overline{x}(k)\| \leq G \left(\sum_{m=0}^{k-1}\alpha_{i1}^2(m)\right)^{\frac{1}{2}}\left(\sum_{m=0}^{k-1}\lambda^{2(k-1-m)}\right)^{\frac{1}{2}}
\end{eqnarray}
where
\[
\left(\sum_{m=0}^k \alpha_i^2(m) \right)^{\frac{1}{2}} \leq \frac{\sqrt{k}}{\mu}
\]

and
\[\left(\sum_{m=0}^{k-1}\lambda^{2(k-1-m)}\right)^{\frac{1}{2}} \leq \left(\dfrac{1}{1-\lambda^{2}}\right)^{\frac{1}{2}} = \dfrac{1}{\sqrt{1-\lambda^{2}}} \triangleq\text{Q}_3.\]
Moreover if  
\begin{equation*}
    \sum_{m{=}0}^{k{-}1} \alpha_{i1}^2(m)\leq \frac{1}{G^2 \sum_{m=0}^{k-1}\lambda^{2(k-1-m)}},
\end{equation*}
then each local agent's estimates converges Q-linearly to its average; that is $\|x_i(k) - \overline{x}(k)\| \leq 1.$

\end{proof}

\subsection{Proof of Lemma \ref{lemmabounds2}}\label{Proof of Lemma 6}
\begin{proof}
We consider the second step size $\alpha_{i2}$ as expressed in equation \eqref{decentralizedstepsize2}, and 
for notation simplicity, we denote $ Z = W \otimes I_p \in \mathbb{R}^{np\times np}$.
Then the distributed iteration at time step $k$ is given by:
\begin{equation}
X(k+1) =  ZX(k) - \alpha_{i2}(k)\nabla f(X(k)).
\label{distributedstepsize2}
\end{equation}
Similarly based on the definition of $\overline{x}(k)$ as the average of local estimates, then we have:
\begin{equation*}
   \overline{x}(k) = \frac{1}{n}\sum\limits_{i=1}^n x_i(k) .
\end{equation*}
Because $g(x(k))$ to be the average of local gradients, So we obtain the expression:
\begin{equation*}
 g(x(k)) = \frac{1}{n}\sum\limits_{i=1}^n \nabla f_i\left(x_i(k)\right).  
\end{equation*}
If we solve equation (\ref{distributedstepsize2}) for $X(k)$ where $\alpha_{i2}$ is the second BB step size, we obtain the expression:
\begin{equation}
X(k) = -\sum_{m=0}^{k-1} \alpha_2(m) \left(W^{(k-1-m)}\otimes I\right)\nabla f\left(x(m)\right).
\label{eq2}
\end{equation}
Let $\overline{X}(k)$ be the average of all concatenated $x_i(k)$, then we obtain:
\begin{equation}
\label{averageofbigX2}
\overline{X}(k) = \dfrac{1}{n}\left(\left(1_n 1_n^T\right) \otimes I\right)x(k).
\end{equation}
Also, we know that the following relationship holds.
\begin{equation}
\|x_i(k) - \overline{x}(k)\| \leq \|X(k) - \overline{X}(k)\|.
\label{eq3}
\end{equation}
We note that in equation (\ref{eq2}), $\alpha_{i2}$ is expressed in equation \eqref{decentralizedstepsize2}.
From equations (\ref{averageofbigX2}) and (\ref{eq3}), we get:
\begin{equation}
\begin{aligned}
\|x_i(k) - \overline{x}(k)\| 
 \leq & \|X(k) - \overline{X}(k)\|, \\
= & \|X(k) -\frac{1}{n}\left(\left(1_n1_n^T\right)\otimes I\right)X(k)\|,\\
 = & \|-\sum\limits_{m=0}^{k{-}1}\alpha_{i2}(m)(W^{(k{-}1{-}m)}\otimes I)\nabla f(x(m)) \\
  & +\frac{1}{n}((1_n1_n^T)\otimes I)\sum_{m{=}0}^{k{-}1}(\alpha_{i2}(m) \\
 & (W^{(k{-}1{-}m)}\otimes I)\nabla f(x(m)))\|, \\
 = & \|-\sum\limits_{m=0}^{k{-}1}(\alpha_{i2}(m)(W^{(k{-}1{-}m)}\otimes I)\nabla f(x(m))) \\ 
  &  +\sum\limits_{m=0}^{k-1}\frac{\alpha_{i2}(m)}{n} ((1_n1_n^T)\otimes I)\nabla f(x(m))\|.
\end{aligned}
\end{equation}
Because $W$ is doubly stochastic, then
\begin{equation*}
\begin{aligned}
\|x_i(k) - \overline{x}(k)\| 
\leq& -\sum_{m=0}^{k-1}\| \alpha_{i2}(m)((W^{(k-1-m)}-\frac{1}{n}1_n1_n^T)\\
& \otimes I)\nabla f(x(m))\|.\nonumber
\end{aligned}
\end{equation*}
If the second BB step size is given by:
\[\alpha_{i2}(k) = \frac{s_i(k-1)^Ty_i(k-1)}{y_i(k-1)^Ty_i(k-1)},\]
then we have the relationship:
\[\alpha_{i2}(m) = \frac{s_i(m-1)^Ty_i(m-1)}{y_i(m-1)^Ty_i(m-1)}.\]
\begin{eqnarray}
\|x_i(k) - \overline{x}(k)\| & \leq & -\sum_{m=0}^{k-1}\frac{s_i(m-1)^Ty_i(m-1)}{\|y_i(m-1)\|^2}\nonumber\\
&& \|W^{(k-1-m)} -  \frac{1}{n}1_n1_n^T\|\|\nabla f(x(m))\|.\nonumber
\end{eqnarray}
We now obtain the bound:\\ $\|x_i(k) - \overline{x}(k)\|$:
\begin{eqnarray}
\leq  \sum_{m=0}^{k-1} \frac{s_i(m-1)^Ty_i(m-1)}{\|y_i(m-1)\|^2}\lambda^{(k-1-m)}\|\nabla f(x(m))\|.\nonumber
\end{eqnarray}
If $\nabla f(x(m))$ is bounded meaning that $\|\nabla f(x(m))\| \leq G$ where $G$ is positive, then we obtain:
\begin{equation}
\|x_i(k) - \overline{x}(k)\|\leq \sum_{m=0}^{k-1} \frac {s_i(m-1)^Ty_i(m-1)}{\|y_i(m-1)\|^2}\lambda^{(k-1-m)}G.
\label{eq5}
\end{equation}
and $\lambda$ is the second largest eigenvalue of W. We also note that the weight matrix $W$ satisfies the inequality, $0<\lambda\leq 1$.
From equation (\ref{eq5}) and by Cauchy-Schwarz on sums, we obtain:
\begin{eqnarray}
\|x_i(k) - \overline{x}(k)\| \leq G\left(\sum_{m=0}^{k-1}\alpha_{i2}^2(m)\right)^{\frac{1}{2}}\left(\sum_{m=0}^{k-1}\lambda^{2(k-1-m)}\right)^{\frac{1}{2}}.
\label{eq6}
\end{eqnarray}
We know that the second BB step size is upper bounded such that $\alpha_i < \frac{1}{\mu}$. \\
By squaring both sides, $\alpha_{i2}^2 \leq \frac{1}{\mu^2}$ and we obtain the relationship $\sum_{m=0}^k \alpha_i^2(m) \leq \frac{k}{\mu^2}$ and we obtain the result $\left(\sum_{m=0}^k \alpha_{i2}^2(m) \right)^{\frac{1}{2}} \leq \frac{\sqrt{k}}{\mu}$.

In equation (\ref{disiteration2}),
\begin{eqnarray}
\|x_i(k) - \overline{x}(k)\| \leq G \left(\sum_{m=0}^{k-1}\alpha_{i2}^2(m)\right)^{\frac{1}{2}}\left(\sum_{m=0}^{k-1}\lambda^{2(k-1-m)}\right)^{\frac{1}{2}},
\end{eqnarray}
where
\[
\left(\sum_{m=0}^k \alpha_{i2}^2(m) \right)^{\frac{1}{2}} \leq \frac{\sqrt{k}}{\mu},
\]

and
\[\left(\sum_{m=0}^{k-1}\lambda^{2(k-1-m)}\right)^{\frac{1}{2}} \leq \left(\dfrac{1}{1-\lambda^{2}}\right)^{\frac{1}{2}} = \dfrac{1}{\sqrt{1-\lambda^{2}}} \triangleq\text{Q}_3.\]
\end{proof}

\subsection{Proof of Lemma \ref{lemmaboundsneighborhood2}}\label{secondstepsizeconvproof}
\begin{proof}
Let the variables $\overline{x}(k)$, $g(k)$, and   $\overline{\alpha_{i2}}(k)$ be defined just as they were in Lemma \ref{lemmaboundsneighborhood1} except that $\overline{\alpha_{i2}}(k)$ is now the average of the second distributed BB step size. We first consider $\|\overline{x}(k+1)-x^*\|^2$ to obtain bounds for convergence. So we start with the iterate:
\begin{equation*}
  \|\overline{x}(k+1)-x^*\|=\|\overline{x}(k)-x^*-\overline{\alpha_{i2}}g(k)\|.
\end{equation*}
 Squaring both sides and simplifying the right hand side of the above equation yields:
\begin{equation}
\begin{aligned}
\|x(k) {-} x^* {-}\overline{\alpha_{i2}}(k)g(k)\|^2 =& \|x(k) {-} x^*\|^2 + \overline{\alpha_{i2}}^2(k)\|g(k)\|^2 \\
 & -2\left(x(k){-}x^*\right)^T\left(\overline{\alpha_{i2}}(k) g(k)\right).
\end{aligned}
\end{equation}
By using vector norm principles,  for vectors $a,\ b, \ 2a^{T}b\leq \|a\|^2+\|b\|^2$, therefore we obtain:
\begin{equation*}
2\left(\overline{x}(k)- x^*\right)^T\left(g(k)\right) \leq \|g(k)\|^2 + \|\overline{x}(k) - x^*\|^2.
\end{equation*}
Just as was defined in Lemma \ref{lemmaboundsneighborhood1}, $\mu$ and $L$ are strong convexity and Lipschitz parameters respectively and $c_1$, $c_2$ are given by $c_1=2/(\mu +L)$ and $c_2=2\mu L/(\mu +L)$. Therefore we have the relationship:
\begin{equation}
\begin{aligned}
\|\overline{x}(k{+}1) {-} x^*{-}\overline{\alpha_{i2}}(k)g(k)\|^2 
 \leq & \|\overline{x}(k){-} x^*\|^2 {+} \overline{\alpha_{i2}}^2(k)\|g(k)\|^2 \\
 & -\overline{\alpha_{i2}}(k)c_1\|g(k)\|^2 \\&- \overline{\alpha_{i2}}(k)c_2\|\overline{x}(k) {-} x^*\|^2, \\
  \leq  & (1-\overline{\alpha_{i2}}(k)c_2)\|\overline{x}(k) {-} x^*\|^2\\ &+(\overline{\alpha_{i2}}^2(k)-\overline{\alpha_{i2}}(k)c_1)\|g(k)\|^2,\\
 \leq & (1-\overline{\alpha_{i2}}(k)c_2)\|\overline{x}(k) {-} x^*\|^2.
 \label{qdistributedrate22}
\end{aligned}
\end{equation}
We note that the last inequality is due to Theorem $2.1.12$ from chapter $2$ of \cite{nesterov1998introductory}. We also note that  $(\overline{\alpha_{i2}}^2(k)-\overline{\alpha_{i2}}(k)c_1)\|g(k)\|^2\leq 0$ provided $\overline{\alpha_{i2}}(k)\leq c_1$. We also note that $\overline{\alpha_{i2}}(k)=c_1$ is within the range of the BB step size bounds below and the details are shown in Appendix \ref{Proof of Cor No11}:

Therefore the distributed BB convergence using the second BB step size can be finalized as:

\begin{eqnarray}
\|\overline{x}(k+1) - x^*\|^2 &\leq& \left(1-\overline{\alpha_{i2}}(k)c_2\right)\|\overline{x}(k) - x^*\|^2,\nonumber\\
 \dfrac{\|\overline{x}(k+1) - x^*\|^2}{\|\overline{x}(k) - x^*\|^2} &\leq& 1-\overline{\alpha_{i2}}(k)c_2,\nonumber\\
\text{and} \quad \dfrac{\|\overline{x}(k+1) - x^*\|}{\|\overline{x}(k) - x^*\|} &\leq& \left(1-\overline{\alpha_{i2}}(k)c_2\right)^{\frac{1}{2}}. \nonumber
\end{eqnarray}
 We will now bound: $\left(1-\overline{\alpha_{i2}}(k)c_2\right)^{\frac{1}{2}}$ where $\alpha_{i2}(k)$ is expressed in equation \eqref{decentralizedstepsize2} and $\overline{\alpha_{i2}}(k)=\frac{1}{n}\sum\limits_{i=1}^n \alpha_{i2}(k)$. 
 By using Lipschitz continuity of $\nabla f(\cdot)$ with $L$ as the Lipschitz constant, the distributed form of the second BB step size is lower bounded by $\frac{1}{L}$.
Now, $\overline{\alpha_{i2}}(k)=\frac{1}{n}\sum\limits_{i=1}^n \alpha_{i2}(k)$, and it results to $n\overline{\alpha_{i2}}(k)=\sum\limits_{i=1}^n \alpha_{i2}(k)$. However, $\alpha_{i2}(k)>\frac{1}{L}$. We know that $\alpha_{i2}(k)<\sum\limits_{i=1}^n \alpha_{i2}(k)$ and we obtain the result:
\begin{equation}\label{sum-step-size-bound-2}
    \frac{1}{L}<\alpha_{i2}(k)<\sum\limits_{i=1}^n \alpha_{i2}(k).
\end{equation}
From equation \eqref{sum-step-size-bound-2}, $n\overline{\alpha_{i2}}(k)=\sum\limits_{i=1}^n \alpha_{i2}(k)>\frac{1}{L}$ and we have $\overline{\alpha_{i2}}(k)>\frac{1}{nL}$.

If $\overline{\alpha_{i2}}(k)$ and $c_2$ are positive and $\overline{\alpha_{i2}}(k) > 1/nL$, then
 $ -\overline{\alpha_{i2}}(k)c_2 < -c_2/nL$.
  Therefore, $\overline{\alpha_{i2}}(k) > 1/nL$.
it implies that $ 0 < 1-\overline{\alpha_{i2}}(k)c_2 < 1-c_2/nL$,
So we obtain the bounds:
\begin{eqnarray}
 0 < \left(1-\overline{\alpha_{i2}}(k)c_2\right)^{\frac{1}{2}} < \left(1-\dfrac{c_2}{nL}\right)^{\frac{1}{2}}.\nonumber
\end{eqnarray}
It has been established in Lemma \ref{lemmaboundsneighborhood1} that $c_2/nL\leq1$ for all positive values of $n$.
Therefore, $$\lim_{k\to\infty}\quad \dfrac{\|\overline{x}(k+1) - x^*\|}{\|\overline{x}(k) - x^*\|} \leq\left(1-\dfrac{c_2}{nL}\right)^{\frac{1}{2}}\leq 1,$$
So the average of the estimates converges Q-linearly to the optimal point, $x^*$.
\end{proof}
\end{document}